\documentclass{amsart}

\usepackage{amssymb}
\usepackage{bm}
\usepackage{graphicx}
\usepackage{psfrag}
\usepackage{color}
\usepackage{url}
\usepackage{amsmath,color}
\usepackage{enumerate}
\usepackage{algpseudocode}
\usepackage{lmodern}
\usepackage{todonotes}
\usepackage{xy}
\input xy
\xyoption{all}

\newcommand{\excise}[1]{}%{$\star$\textsc{#1}$\star$}

\newcommand{\ve}{\boldsymbol}

\newtheorem{thm}{Theorem}
\newtheorem{lemma}[thm]{Lemma}
\newtheorem{cor}[thm]{Corollary}
\newtheorem{prop}[thm]{Proposition}

\theoremstyle{definition}

{\bfseries}{\normalfont}
{\bfseries}{\normalfont}

{\bfseries}{\normalfont}

\newcommand{\be}{\begin{eqnarray}}
\newcommand{\bea}{\begin{eqnarray*}}
\newcommand{\ee}{\end{eqnarray}}
\newcommand{\eea}{\end{eqnarray*}}

\newcommand{\ring}[1]{\ensuremath{\mathbb{#1}}}

\newcommand{\conv}{\mathrm{conv}}
\newcommand{\vol}{\mathrm{vol}}
\newcommand{\rank}{\mathrm{rank}}
\newcommand{\R}{\mathbb R}
\newcommand{\Z}{\mathbb Z}

\newcommand{\sg}{{Sg}}

\newcommand\sail{{E}}
\newcommand{\Sp}{{H}}
\newcommand{\KP}{{P}}

\newcommand\CP{{C}}

\newcommand\ZZ{\ring{Z}}

\newcommand\xx{{\mathbf x}}
\newcommand\yy{{\mathbf y}}

\newcommand\bb{{\mathbf b}}
\newcommand\vv[1]{{\mathbf{#1}}}

\newcommand\qq{{\boldsymbol q}}

\newcommand\supp{\mathrm{supp}}

\newcommand\icone{{Sg}}

\newcommand{\modulo}{\,\mathrm{mod}\;}

\global\long\def\xx{\ve x}
\global\long\def\yy{\ve y}
\global\long\def\bb{\ve b}
\global\long\def\cc{\ve c}

\global\long\def\uu{\ve u}
\global\long\def\vv{\ve v}
\global\long\def\zero{\mathbf{0}}
\global\long\def\ol#1{\bar{#1}}
\global\long\def\zz{\ve z}
\global\long\def\oneone{\ve 1}
\global\long\def\one{\mathbf{1}}
\global\long\def\t{\top}

\def\udots{\mathinner{\mkern1mu\raise1pt\vbox{\kern7pt\hbox{.}}\mkern2mu
\raise4pt\hbox{.}\mkern2mu\raise7pt\hbox{.}\mkern1mu}}

\def\ol#1{{\overline {#1}}}

\title[Distance-sparsity transference]{Distance-sparsity transference for vertices of  corner polyhedra}
\author[I. Aliev, M. Celaya, M. Henk and A. Williams]{Iskander Aliev, Marcel Celaya, Martin Henk and Aled Williams}

\date{\today }

\begin{document}
\maketitle

\begin{abstract}
\noindent

We obtain a transference bound for vertices of corner polyhedra that connects two well-established areas of research: proximity and sparsity of solutions to integer programs. In the knapsack scenario, it implies that for any vertex ${\ve x}^*$ of an integer feasible knapsack polytope $\KP({\ve a}, { b})=\{{\ve x}\in \R^n_{\ge 0}: {\ve a}^T{\ve x}={ b}\}$, ${\ve a}\in \Z^n_{>0}$, there exists an integer point ${\ve z}^*\in \KP({\ve a}, { b})$ such that, denoting by $s$ the size of support of ${\ve z}^*$ and assuming $s>0$, % that satisfies the optimal upper bound 
\begin{equation}\label{abstract}
\|{\ve x}^{*}-{\ve z}^{*}\|_{\infty} \,\frac{2^{s-1}}{s} < \|{\ve a}\|_{\infty}\,,
\end{equation}
where $\|\cdot\|_{\infty}$ stands for the $\ell_{\infty}$-norm. The bound \eqref{abstract} gives an exponential in $s$ improvement on previously known proximity estimates. In addition, for general integer linear programs we obtain a resembling result that connects the minimum absolute nonzero entry of an optimal solution with the size of its support.

\end{abstract}

\section{Introduction and statement of results}

The main contribution of this paper shows a surprising relation that holds between two well-established areas of research, proximity and sparsity of solutions to integer programs, in the case of  Gomory's corner polyhedra.

The proximity-type results provide estimates for the distance between optimal vertex solutions of linear programming relaxations and feasible integer points, with seminal works by Cook et al. \cite{MR839604} and, more recently, by Eisenbrand and Weismantel \cite{MR3775840}.
The sparsity-type results, in their turn, provide bounds for the size of support of feasible integer points and solutions to integer programs. Bounds of this type are dated back to the classical integer Carath\'eodory theorems of Cook, Fonlupt and Schriver \cite{MR830593} and Seb\H{o} \cite{Sebo}. More recent contributions include results of Eisenbrand and Shmonin \cite{EisenbrandShmonin2006} and Aliev et al. \cite{ADON2017,Support, AADO}.  Further, in a very recent work Lee, Paat, Stallknecht and Xu \cite{LPSX} apply new sparsity-type bounds to refine the bounds for proximity. 

To state the main results of this paper, we will need the following notation. Let $A \in \Z^{m\times n}$, $m<n$,
 and let $\tau=\{i_1,\ldots,i_k\} \subseteq  \{1,\ldots,n\}$ with $i_1<i_2<\cdots<i_k$. We will use the notation  $A_\tau$ for the $m \times k$ submatrix  of $A$ with columns indexed by $\tau$.  In the same manner, given ${\ve x}\in \R^n$, we will denote by ${\ve x}_{\tau}$ the vector $(x_{i_1}, \ldots, x_{i_k})^{\t}$. The complement of $\tau$ in $ \{1,\ldots, n\}$ will be denoted as $\bar\tau$. We will say that $\tau$ is a {\em basis} of $A$ if $|\tau|=m$ and the submatrix $A_{\tau}$  is nonsingular. By $\Sigma(A)$ we will denote the maximum absolute $m\times m$ subdeterminant of $A$:
\bea
\Sigma(A)=\max\{|\det(A_{\tau})|: \tau\subset \{1,\ldots, n\} \mbox{ with } |\tau|= m \}\,.
\eea
When $\Sigma(A)$ is positive, $\gcd(A)$ will denote the greatest common divisor of all $m \times m$ subdeterminants of $A$.

For ${\ve x}=(x_1, \ldots, x_n)^{\t}\in \R^n$, the $\ell_{\infty}$-norm of ${\ve x}$ will be denoted as $\|{\ve x}\|_{\infty}$.   We will denote by $\supp({\ve x})=\{i : x_i \not= 0\}$ the {\em support} of ${\ve x}$. Further, $\|{\ve x}\|_0:=|\supp({\ve x})|$ will denote the $0$-``norm'', widely used in the theory of {compressed sensing} \cite{candestao,CS_survey}, which counts the cardinality of the support of ${\ve x}$. 

Let $A \in \Z^{m\times n}$ with $m<n$ and $\ve b \in \Z^m $. We will consider the polyhedron
\bea
\KP(A, {\ve b})=\{{\ve x}\in \R^n_{\ge 0}: A{\ve x}={\ve b}\}
\eea
and take any vertex $\ve x^*$  of $\KP(A, {\ve b})$.  %
Without loss of generality, we may assume that
\be\label{matrix_partition}
A=(A_{\gamma}, A_{\bar \gamma})\in \Z^{m\times n} \mbox{ with nonsingular } A_{\gamma}\,,
\ee
\be\label{vertex_assumption}
 {\ve x}^*_\gamma = A_\gamma^{-1}{\ve b} \mbox{ and }{\ve x}^*_{\bar \gamma}={\ve 0}\,.
\ee
In general, there can be many choices for the basis $\gamma$. However, if ${\ve x}^*$ is \emph{nondegenerate}; that is, if the size of the support of ${\ve x}^*$ is $m$, then there is a unique choice for $\gamma$, namely $\gamma=\supp({\ve x})$.  For convenience, throughout this paper we will assume, given a choice for $\gamma$, that
\bea \gamma= \{1, \ldots, m\}\,.\eea
For a set ${S}\subset \R^n$ we will denote by $\conv({S})$ the {\em convex hull} of ${S}$.  Gomory  \cite{Gomory_polyhedra} introduced  the {\em corner polyhedron} $\CP_{\gamma}(A, {\ve b})$ associated with $\gamma$ as
\bea\label{corner_polyhedron}
\CP_{\gamma}(A, {\ve b})=\conv (\{{\ve x}\in \Z^n: A{\ve x}={\ve b},  {\ve x}_{\bar \gamma}\ge 0\})\,.
\eea
\begin{thm}\label{t:IPTP_general}
Suppose that  $A$ satisfies \eqref{matrix_partition}, ${\ve b}\in \Z^m$,  $\ve x^*$ is given by \eqref{vertex_assumption} and the corner polyhedron $\CP_{\gamma}(A, {\ve b})$ is nonempty. 
Let ${\ve z}^{*}$  be a vertex of  $\CP_{\gamma}(A, {\ve b})$ and let $r=\|{\ve z}^{*}_{\bar\gamma}\|_0$. Then 
\begin{equation}\label{eq:r_equals_0}
{\ve x}^{*}={\ve z}^{*} \mbox{ if } r= 0\,,
\end{equation}
\begin{equation}\label{eq:r_equals_1}
\|{\ve x}^{*}-{\ve z}^{*}\|_{\infty} \le \frac{\Sigma(A)}{\gcd(A)}-1 \mbox{ if } r= 1 \mbox{ and }
\end{equation}
\begin{equation}\label{eq:r_greater_1}
\|{\ve x}^{*}-{\ve z}^{*}\|_{\infty} \,\frac{2^{r}}{r} \leq \frac{\Sigma(A)}{\gcd(A)} \mbox{ if } r\ge 2\,.
\end{equation}
\end{thm}

The bounds \eqref{eq:r_equals_1} and \eqref{eq:r_greater_1} are optimal. Specifically, \eqref{eq:r_equals_1} is attained already in the knapsack scenario (with the choice of parameters  
\eqref{optimal_example_r_equals_1}). The bound \eqref{eq:r_greater_1}, in its turn, is attained for 
\bea
A= \left(\begin{array}{cccc}
2 & 0 & 5 & 5\\
0 & 4 & 2 &-1
\end{array}
\right)\,,\;
{\ve b}=\left(\begin{array}{c}
20\\
3
\end{array}
\right)\mbox{ and }
\eea
the vertex ${\ve x}^*=(10, 3/4, 0, 0)^{\t}$ of $\KP(A, {\ve b})$. For this choice of parameters the corner polyhedron $\CP_{\gamma}(A, {\ve b})$
has the unique vertex  ${\ve z}^{*}= (0,1,1,3)^{\t}$.

Theorem \ref{t:IPTP_general} shows that for the corner polyhedron associated with a vertex ${\ve x}^*$ of $\KP(A, {\ve b})$ a strong {\em proximity-sparsity transference} holds: the distance from ${\ve x}^*$ to any vertex ${\ve z}^*$ of the corner polyhedron {\em exponentially} drops with the size of support of ${\ve z}^*$ and, vice versa, the size of support of ${\ve z}^*$ reduces with the growth of its distance to ${\ve x}^*$ .

Suppose next that the polyhedron $\KP(A, {\ve b})$ is integer feasible and consider its {\em integer hull} $\KP_I= \conv(\KP(A, {\ve b})\cap \Z^n)$.
A natural direction for a further research would be to derive a distance-sparsity transference bound for the vertices of $\KP_I$. Notice that the set $\KP(A, {\ve b})\cap \Z^n$ is obtained from $\CP_{\gamma}(A, {\ve b})\cap\Z^n$ by enforcing back the nonnegativity constraints  $ {\ve x}_{\gamma}\ge 0$ and this may potentially result in cutting off all vertices of the corner polyhedron. In Section \ref{knapsackintro} we show that in the knapsack scenario at least one vertex of $\CP_{\gamma}(A, {\ve b})$ avoids the cut and Theorem \ref{t:IPTP_general} implies an optimal  distance-sparsity transference bound for lattice points in the knapsack polytope. 

Although it remains an open problem to extend Theorem \ref{t:IPTP_general} to vertices of  $\KP_I$ in the general setting, the next result of this paper allows enforcing back the constraints $ {\ve x}_{\gamma}\ge 0$ that are tight at ${\ve x}^*$ when $\tau:=\supp({\ve x}^*)$ has size strictly less than $m$. In this situation, following \cite{thomas2005structure}, we may define the polyhedron
\bea
\CP_{\tau}(A,\bb)=\conv(\{\xx\in\Z^{n}:A\xx=\bb,\xx_{\bar{\tau}}\geq\zero\}).
\eea
In this setting the choice of basis $\gamma$ in \eqref{matrix_partition} is typically not unique. However, we show is that there exists at least one basis $\gamma$ for which the conclusions of Theorem~\ref{t:IPTP_general} remain valid for this polyhedron, up to a factor which depends on the number of zero coordinates of $\xx_{\tau}^*$.

\begin{thm}
\label{thm:upper_bound_extension}
Suppose that $A$ satisfies \eqref{matrix_partition} and ${\ve b}\in \Z^m$. Let  $\ve x^*$ be given by \eqref{vertex_assumption} and let $\tau=\supp(\xx^{*})$. Let $\zz^{*}$ be a vertex of the polyhedron $\CP_{\tau}(A,\bb)$. Then the basis $\gamma$ in \eqref{matrix_partition} may be chosen so that, letting $r=\|{\ve z}^{*}_{\bar\gamma}\|_0$ and $d=m-\left|\tau\right|$, we have
\begin{equation}\label{eq:r_equals_0_extension}
{\ve x}^{*}={\ve z}^{*} \mbox{ if } r= 0\,,
\end{equation}
\begin{equation}\label{eq:r_equals_1_extension}
\|{\ve x}^{*}-{\ve z}^{*}\|_{\infty} \le \frac{\Sigma(A)}{\gcd(A)}-1 \mbox{ if } r= 1 \mbox{ and }
\end{equation}
\begin{equation}\label{eq:r_greater_1_extension}
\|{\ve x}^{*}-{\ve z}^{*}\|_{\infty} \,\frac{2^{r}}{r^{d+1}} \leq \frac{\Sigma(A)}{\gcd(A)} \mbox{ if } r\ge 2\,.
\end{equation}
\end{thm}
We remark that integer programs of the form $A\xx=\bb,\xx_{\bar{\tau}}\geq\zero, \xx\in\Z^n$, where $\tau$ is the support of a vertex of $\KP(A, {\ve b})$, have been investigated in \cite{thomas2005structure}. Such an integer program is called a \emph{Gomory relaxation with respect to $\tau$}. See in particular \cite[\S 2]{thomas2005structure} for more details.

\subsection{Distance-sparsity transference for knapsacks}\label{knapsackintro}

We will now separately consider the case $A\in \Z_{>0}^{1\times n}$, known as {\em knapsack scenario}. We will follow a traditional vector notation and replace $A$ and ${\ve b}$ with a positive integer vector ${\ve a}=(a_1, \ldots, a_n)^{\t} \in \Z^{n}_{>0}$ and integer $b\in\Z$.  
In this setting $\KP(A, {\ve b})$ is referred to as the  {\em knapsack polytope}
\bea
\KP({\ve a},b)=\{{\ve x}\in \R^n_{\ge 0}: {\ve a}^{\t}{\ve x}=b\}\,.
\eea
In what follows, we will exclude the trivial case $n=1$ and assume that $n\ge 2$. We also assume that the polytope $\KP({\ve a}, b)$ contains integer points. Equivalently, ${b}$ belongs to the \emph{semigroup}
\begin{equation*}
\icone({\ve a}) = \{{\ve a}^{\t}{\ve z}: {\ve z} \in \ZZ^n_{\ge 0} \}
\end{equation*}
generated by the entries of the vector ${\ve a}$. Note that any element of the semigroup $\icone({\ve a}) $ must be divisible by the greatest common divisor $\gcd(a_1, \ldots, a_n)$ of $a_1, \ldots, a_n$. Hence, we may assume without loss of generality that ${\ve a}$  satisfies the following conditions:
\be \label{positive_primitive}
\begin{array}{l}
{\ve a}=(a_1, \ldots, a_n)^{\t} \in \Z^{ n}_{>0}\,,n\ge 2\,, \text{ and }
\gcd(a_1, \ldots, a_n)=1\,.
\end{array}
\ee

Aliev et al \cite[Theorem 1] {AHO} proved that for any vertex ${\ve x}^{*}$ of the polytope $\KP({\ve a},b)$ there exists an integer point ${\ve z}\in  \KP({\ve a},b)$ such that 
\be\label{AHO_bound}
\|{\ve x}^{*}-{\ve z}\|_{\infty}\le \|{\ve a}\|_\infty -1
\ee
and that the bound (\ref{AHO_bound}) is sharp in the following sense.
For any positive integer $k$ and any dimension $n$ there exist ${\ve a}$ satisfying (\ref{positive_primitive}) with $\|{\ve a}\|_{\infty}=k$  and $b\in \Z$ such that
the knapsack polytope $\KP({\ve a},b)$ contains exactly one integer point ${\ve z}$ and $\|{\ve x}^{*}-{\ve z}\|_{\infty}=\|{\ve a}\|_\infty-1$.

The best known sparsity-type estimate in the knapsack scenario \eqref{positive_primitive}, obtained in  \cite[Theorem 6]{AADO}, guarantees existence of an integer point  ${\ve z}\in \KP({\ve a},b)$ that satisfies the bound
 \be\label{sparse_knapsack}
 \|{\ve z}\|_0\le 1 + \log(\min\{a_1, \ldots, a_n\})\,.
 \ee
The next result will combine and refine the bounds (\ref{AHO_bound}) and (\ref{sparse_knapsack}) as follows. 
The knapsack polytope $\KP({\ve a}, b)$ is an $(n-1)$-dimensional simplex in $\R^n$ with vertices $({b}/{a_1})\ve{e}_1, \ldots, ({b}/{a_n})\ve{e}_n$,
where $\ve{e}_i$ denotes the $i$-th standard basis vector. 
Hence, without loss of generality, we may assume that the vertex ${\ve x}^{*}$ has the form 
\be\label{knapsack_vertex} {\ve x}^{*}=\frac{b}{a_1}{\ve e}_1\,.
\ee
The corner polyhedron associated with the vertex ${\ve x}^{*} $ (with $\gamma=\{1\}$) can be written as
$\CP_{\gamma}({\ve a}, {b})=\conv (\{{\ve x}\in \Z^n: {\ve a}^{\t}{\ve x}={ b},  x_2\ge 0, \ldots, x_n\ge 0\})$.
%
%
%%%%%%%%%%%%%%%%%%%%%%%%%%%%%%%%%%%
%
\begin{thm}\label{t:IPTP_positive_knapsack}
Let ${\ve a}$ satisfy (\ref{positive_primitive}), $b\in\icone({\ve a})$ and ${\ve x}^{*}$ is given by \eqref{knapsack_vertex}. Then $ \KP({\ve a}, {b})$ contains a vertex ${\ve z}^{*}$ of $\CP_{\gamma}({\ve a}, {b}) $ with $r=\|(z^*_2,\ldots, z^*_n)^{\t}\|_0$  such that 
\begin{equation}\label{eq:r_equals_0_knapsack}
{\ve x}^{*}={\ve z}^{*} \mbox{ if } r= 0\,,
\end{equation}
\begin{equation}\label{eq:r_equals_1_knapsack}
\|{\ve x}^{*}-{\ve z}^{*}\|_{\infty} \le \|{\ve a}\|_{\infty}-1 \mbox{ if } r= 1 \mbox{ and }
\end{equation}
\begin{equation}\label{eq:r_greater_1_knapsack}
\|{\ve x}^{*}-{\ve z}^{*}\|_{\infty} \,\frac{2^{r}}{r} < \|{\ve a}\|_{\infty}  \mbox{ if } r\ge 2\,.
\end{equation}
\end{thm}
%
%%%%%%%%%%%%%%%%%%%%%%%%%%%%%%%%%%%
%
%
%
Theorem \ref{t:IPTP_positive_knapsack} can be viewed as a transference result that allows strengthening the distance bound (\ref{AHO_bound})  if integer points in the knapsack polytope  are not sparse and, vice versa, strengthening the sparsity bound  (\ref{sparse_knapsack}) if feasible integer points are sufficiently far from a vertex of the knapsack polytope.

 Given a cost vector ${\ve c}\in \Z^n$, we will now consider the integer knapsack problem
 \be
\min\{ {\ve c}^{\t}{\ve x}: {\ve x}\in \KP({\ve a},b)\cap \Z^n\}\,.
\label{initial_IP}
\ee
Note that  \eqref{initial_IP} is feasible since $b\in\sg({\ve a})$. 

Let $IP({\ve c}, {\ve a}, {b})$ and $LP({\ve c}, {\ve a}, { b})$ denote the optimal values of (\ref{initial_IP}) and its linear programming relaxation
\be\label{initial_LP}
\min\{ {\ve c}^{\t} {\ve x}: {\ve x}\in \KP({\ve a},b)\}\,,
\ee
 respectively.
The {\em integrality gap } $IG({\ve c}, {\ve a}, {b})$ of  (\ref{initial_IP}) is defined as
\bea IG({\ve c}, {\ve a}, {b})= IP({\ve c}, {\ve a}, { b})-LP({\ve c}, {\ve a}, { b})\,.\eea
As a corollary of Theorem \ref{t:IPTP_positive_knapsack}, we obtain the following bound for the integrality gap.
\begin{cor}\label{coro_upper_bound} Let ${\ve a}$ satisfy (\ref{positive_primitive}), $b\in\icone({\ve a})$ and ${\ve c}\in\Z^n$.
Suppose that ${\ve x}^{*}$ given by \eqref{knapsack_vertex} is an optimal vertex solution to  \eqref{initial_LP}. 
Let further ${\ve z}^{*}$ be any vertex of $\CP_{\gamma}({\ve a}, {b}) $ such that ${\ve z}^{*}\in \KP({\ve a}, b)$. Then for $r=\|(z^*_2,\ldots, z^*_n)^{\t}\|_0$  we have
\begin{equation}\label{eq:r_equals_0_gap}
IG({\ve c}, {\ve a}, {b})=0 \mbox{ if } r= 0\,,
\end{equation}
\begin{equation}\label{eq:r_equals_1_gap}
IG({\ve c}, {\ve a}, {b}) \le 2(\|{\ve a}\|_{\infty}-1) \|{\ve c}\|_{\infty}  \mbox{ if } r=1 \mbox{ and }
\end{equation}
\begin{equation}\label{eq:r_greater_1_gap}
IG({\ve c}, {\ve a}, {b})  < \frac{r(r+1)}{2^{r}}\|{\ve a}\|_{\infty} \|{\ve c}\|_{\infty} \mbox{ if } r\ge 2\,.
\end{equation}
\end{cor}

It follows from the proof of Theorem 1(ii) in \cite{AHO} that the bound \eqref{eq:r_equals_1_knapsack} (and hence \eqref{eq:r_equals_1} and \eqref{eq:r_equals_1_extension} for $m=1$) corresponding to the case $r=1$ is optimal.  For completeness, we recall that it is sufficient to choose a positive integer $k$ and set 
\be\label{optimal_example_r_equals_1}
{\ve a}=(k,\ldots,k,1)^{\t}\,, b=k-1 \mbox{ and } {\ve x}^*=\frac{k-1}{k}\cdot\ve{e}_1\,. 
\ee
Then the knapsack polytope $\KP({\ve a},b)$ contains precisely one integer point,  ${\ve z}^*=(k-1)\cdot\ve{e}_n$ and we obtain
$\|{\ve x}^*-{\ve z}^*\|_{\infty}=k-1 =\|{\ve a}\|_{\infty}-1$.
The next result of this paper shows that the bounds in Theorems \ref{t:IPTP_general} - \ref{t:IPTP_positive_knapsack} are optimal in the knapsack scenario for $r\ge 2$.
%
%
%
%%%%%%%%%%%%%%%%%%%%%%%%%%%%%%%%%%%
%
\begin{thm}\label{t:IPTP_positive_knapsack_sharpness}
Fix integer $s\ge 3$. For any $\epsilon>0$ there exists an integer vector ${\ve a}\in \Z^s$ satisfying (\ref{positive_primitive}) and $b\in\icone({\ve a})$ such that for ${\ve x}^{*}=({b}/{a_1}){\ve e}_1$ the knapsack polytope $ \KP({\ve a}, {b})$ contains a vertex ${\ve z}^*$ of $\CP_{\gamma}({\ve a}, {b}) $ with $\|{\ve z}^*_{\bar\gamma}\|_0=s-1$ and
\begin{equation}\label{eq:general_Knapsack_Bound_via_determinant_sharpness}
\|{\ve x}^*-{\ve z}^*\|_{\infty} \,\frac{2^{s-1}}{s-1} >(1-\epsilon)\|{\ve a}\|_{\infty}\,.
\end{equation}
\end{thm}
%
%%%%%%%%%%%%%%%%%%%%%%%%%%%%%%%%%%%
%
%
%

\subsection{A refined sparsity-type bound for solutions to integer programs}

The next result of this paper aims to refine the general sparsity-type bound obtained in \cite[Theorem 1]{Support}. 
Let
\bea
\rho({\ve x})=\min\{|x_i| : i\in\supp({\ve x}) \}\,
\eea
denote the minimum absolute nonzero entry of ${\ve x}$. The notation $\log(\cdot)$ will be used for logarithm with base two.
Let $A \in \Z^{m\times n}$, $\ve b \in \Z^m $ and $\ve c \in \Z^n$. We will consider  the general integer linear problem in standard form
\begin{equation}
  \label{eq:mainProblem}
  \max\left\{ \ve c^{\t}\ve x : {\ve x}\in \KP(A, {\ve b})\cap  \Z^n\right\}\,.
\end{equation}
We assume that $ \KP(A, {\ve b})$ contains integer points, so that \eqref{eq:mainProblem} is feasible.
We will also assume without loss of generality that the matrix $A$ has full row rank, i.e., $\rank(A)=m$.

It was shown in \cite[Theorem 1]{Support} that there exists an optimal solution $\ve z^*$ for \eqref{eq:mainProblem} satisfying the bound
\begin{equation}\label{t:Bound_via_determinant_bound1}
\|\ve z^*\|_0 \le m + \log\left(\frac{\sqrt{\det(AA^{\t})}}{\gcd(A)}\right)\,.
\end{equation}

%For a set of indices $S=\{i_1,\ldots, i_k\}\subset [n]$ with $1\le i_1<\cdots<i_k\le n$ we will denote
%by $A_S$ the submatrix of $A$ whose columns are labeled by the indices from $S$. Similarly, given ${\ve x}=(x_1, \ldots, x_n)^{\t}\in \R^n$ we will use the notation ${\ve x}_S=(x_{i_1}, \dots, x_{i_k})^{\t}$. 
%
Note that any vertex  solution for \eqref{eq:mainProblem} has the size of support $\le m$. Any non-vertex  solution $\ve z^*$, in its turn, belongs to the interior 
of the face ${\mathcal F}=\KP(A, {\ve b})\cap\{{\ve x}\in \R^n: x_i=0 \mbox{ for } i\notin \supp(\ve z^*)\}$ of the polyhedron $\KP(A, {\ve b})$.
Then the minimum absolute nonzero entry $\rho(\ve z^*)$ is the $\ell_{\infty}$-distance from $\ve z^*$ to the boundary of  ${\mathcal F}$. 
To obtain a refinement of the bound (\ref{t:Bound_via_determinant_bound1}) we will link the minimum absolute nonzero entry and the size of support of solutions to \eqref{eq:mainProblem}.

\begin{thm}
\label{thm:UpperBound}
Let $A \in \Z^{m\times n}$, $\ve b \in \Z^m $, $\ve c \in \Z^n$ and suppose that  \eqref{eq:mainProblem} is feasible. Then there is an optimal solution $\ve z^*$ to \eqref{eq:mainProblem} such that, letting $s=\|{\ve z}^*\|_0$,
\begin{equation}\label{t:Bound_via_transference}
(\rho({\ve z}^*)+1)^{s-m} \le \frac{\sqrt{\det(AA^{\t})}}{\gcd(A)}\,.
\end{equation}
\end{thm}

\section{Lattices and corner polyhedra}

For linearly independent ${\ve b}_1, \ldots, {\ve b}_l$ in $\R^d$, the set $\Lambda=\{\sum_{i=1}^{l} x_i {\ve b}_i,\, x_i\in \Z\}$ is a $l$-dimensional {\em lattice} with {\em basis} ${\ve b}_1, \ldots, {\ve b}_l$ and {\em determinant}
$\det(\Lambda)=(\det({\ve b}_i\cdot {\ve b}_j)_{1\le i,j\le l})^{1/2}$, where ${\ve b}_i\cdot {\ve b}_j$ is the standard inner product of the basis vectors ${\ve b}_i$ and  ${\ve b}_j$. Recall that the {\em Minkowski sum} $X+Y$ of the sets $X, Y\subset \R^d$ consists of all points ${\ve x}+{\ve y}$ with ${\ve x}\in X$ and ${\ve y}\in Y$. The {\em difference set} $X-X$ is the Minkowski sum of $X$ and $-X$.
For a lattice $\Lambda\subset\R^d$ and ${\ve y}\in \R^d$, the set ${\ve y}+\Lambda$ is an {\em affine lattice} with determinant $\det(\Lambda)$.
%Further, for ${\ve x}, {\ve y}\in \R^d$ we write ${\ve x}\equiv {\ve y} ( \modulo \Lambda)$
%if ${\ve x}-{\ve y}\in \Lambda$. 

Let $\Lambda\subset \Z^d$ be a $d$-dimensional integer lattice.
The point ${\ve x}\in \Z^d_{\ge 0}$ is  called {\em irreducible (with respect to $\Lambda$)}  if for any two points ${\ve y}, {\ve y}'\in \Z^d_{\ge 0}$
with $0\le y_i\le x_i$, $0\le y_i'\le x_i$, $i\in \{1,\ldots, d\}$ the inclusion ${\ve y}-{\ve y}'\in \Lambda$ implies ${\ve y}={\ve y}'$.

\begin{lemma}[Theorem 1 in \cite{Gomory_polyhedra}]\label{t:bound_for_irreducible_points}
If ${\ve x}\in \Z^d_{\ge 0}$ is irreducible with respect to the lattice $\Lambda$ then
\be\label{product_det}
\prod_{i=1}^d (x_i+1) \le \det(\Lambda)\,.
\ee

\end{lemma}

\begin{proof}
The lattice $\Lambda$ can be viewed as a subgroup of the additive group $\Z^d$. 
The number of points ${\ve y}\in \Z^d_{\ge 0}$ with $0\le y_i\le x_i$, $i\in\{1,\ldots, d\}$ is equal to $\prod_{i=1}^d (x_i+1)$.
 Since ${\ve x}$ is irreducible, each such ${\ve y}$ corresponds to a unique coset (affine lattice) ${\ve y}+\Lambda$ of $\Lambda$.
Finally notice that there are only $\det(\Lambda)$ different cosets.
\end{proof}

 Let ${\ve r}\in \Z^d$ and consider the affine lattice  $\Gamma= {\ve r}+\Lambda$. We will call the set $\sail(\Gamma)= \conv(\Gamma\cap \R^d_{\ge 0})$ the {\em sail} associated with $\Gamma$. 

\begin{lemma}\label{t:bound_for_irreducible_points}
Every vertex of the sail $\sail(\Gamma)$ is irreducible.
\end{lemma}

\begin{proof}

Let ${\ve x}$ be a vertex of $\sail(\Gamma)$. Suppose, to derive a contradiction, that ${\ve x}$ is reducible. Then
there are distinct points ${\ve y}, {\ve y}'\in \Z^d_{\ge 0}$
with $0\le y_i\le x_i$, $0\le y_i'\le x_i$, $i\in  \{1,\ldots, d\}$ such that ${\ve y}-{\ve y}'\in \Lambda$.

Since ${\ve x}-{\ve y}\in \Z^d_{\ge 0}$ and ${\ve x}-{\ve y}'\in \Z^d_{\ge 0}$, the vectors ${\ve v}_1= {\ve x}-{\ve y}+{\ve y}'$ and ${\ve v}_2= {\ve x}-{\ve y}'+{\ve y}$
have nonnegative integer entries. Further, ${\ve v}_1, {\ve v}_2\in \Gamma$ and  ${\ve x}= ({\ve v}_1+{\ve v}_2 )/2$. Therefore ${\ve x}$ is not a vertex of  $\sail(\Gamma)$.
\end{proof}

\begin{lemma}\label{t:sum_product}
For $d\ge 2$ and $x_1, \ldots, x_d\ge 1$ the inequality
\be\label{sum_product}
x_1+\cdots+x_d \le \frac{d(x_1+1)\cdots(x_d+1)}{2^{d}}
\ee
holds. \end{lemma}

\begin{proof}

Suppose that (\ref{sum_product}) is satisfied for $x_1=y_1, \ldots, x_d=y_d$. We will first show that for any $\epsilon>0$ and any $i\in\{1,\ldots, d\}$ the inequality
(\ref{sum_product}) is satisfied for  $x_1=y_1, \ldots, x_{i-1}=y_{i-1}, x_i=y_i+\epsilon,  x_{i+1}=y_{i+1}, \ldots, x_d=y_d$. After possible renumbering, it is sufficient to consider the case $i=1$. We have
\bea
(y_1+\epsilon)+y_2+\cdots+y_d \le \frac{d(y_1+1)\cdots(y_d+1)}{2^{d}} + \epsilon \\
\le \frac{d(y_1+1)\cdots(y_d+1)}{2^{d}} + \epsilon \frac{d(y_2+1)\cdots(y_d+1)}{2^{d}} = \frac{d(y_1+1+\epsilon)\cdots(y_d+1)}{2^{d}}\,.
\eea
To complete the proof it is sufficient to observe that (\ref{sum_product}) holds for $y_1=\cdots =y_d=1$.
\end{proof}

%%%%%%%%%%%%%%%%%%%%%%%%%%%%%%%%%%%%%%%%%%%%%%%%%%%%%%%%%%%%%%%%%%%%%%%%%%%%%

Given $A\in\Z^{m\times n}$ and ${\ve b}\in\Z^m$, we will denote by $\Gamma(A,{\ve b})$  the set of integer points in the affine subspace
\bea
\Sp({A,{\ve b}})= \{{\ve x}\in \R^n: A{\ve x}={\ve b}\}\,,
\eea
 that is
\bea \Gamma(A,{\ve b})=\Sp({A,{\ve b}})\cap \Z^n\,.\eea
The set $\Gamma(A,{\ve b})$  is an affine lattice of the form $\Gamma(A,{\ve b})= {\ve r}+ \Gamma(A)$,
where ${\ve r}$ is any integer vector with $A{\ve r}={\ve b}$  and $\Gamma(A) = \Gamma(A,{\ve 0})$ is the  lattice formed by all integer points in the kernel of the matrix $A$.

Let $\pi_{\gamma}$ denote the projection map from $\R^n$ to $\R^{n-m}$ that forgets the first $m$ coordinates, $\pi_{\gamma}: {\ve u}\mapsto {\ve u}_{\bar\gamma}$ . Recall that  $A_{\gamma}$ is nonsingular. It follows that  the restricted map $\pi_{\gamma} |_{\Sp({A,{\ve b}})}:\Sp({A,{\ve b}})\rightarrow \R^{n-m} $ is bijective. Specifically, for any ${\ve u}_{\bar\gamma}\in \R^{n-m}$ we have
\be\label{structure_of_projection}
\pi_{\gamma} |_{\Sp({A,{\ve b}})}^{-1}({\ve u}_{\bar\gamma})= \left(
\begin{array}{c}{\ve u}_{\gamma}\\ {\ve u}_{\bar\gamma}\end{array}
\right)\,\mbox{ with } {\ve u}_{\gamma}= A_{\gamma}^{-1}({\ve b}-A_{\bar \gamma}{\ve u}_{\bar\gamma})\,.
\ee

For technical reasons, it is convenient to consider the projected affine lattice $\Lambda(A,{\ve b})=\pi_{\gamma}(\Gamma(A,{\ve b}))$ and the projected lattice $\Lambda(A)=\pi_{\gamma}(\Gamma(A))$.
Let ${\ve g}_1, \ldots, {\ve g}_{n-m}$ be a basis of $\Gamma(A)$.
Since the map $\pi_{\gamma} |_{\Sp({A,{\ve 0}})}$ is bijective, the vectors ${\ve b}_1=\pi_{\gamma}({\ve g}_1), \ldots, {\ve b}_{n-m}=\pi_{\gamma}({\ve g}_{n-m})$ form a basis of the lattice $\Lambda(A)$.
Let $G\in \Z^{n \times (n-m)}$ be the matrix with columns  ${\ve g}_1, \ldots, {\ve g}_{n-m}$.
We will denote by $F$ the $(n-m)\times (n-m)$-submatrix of $G$ consisting of the last $n-m$
rows; hence, the columns of $F$ are ${\ve b}_1, \ldots, {\ve b}_{n-m}$. Then $\det(\Lambda(A))=|\det(F)|$.
The rows of the matrix $A$ span the $m$-dimensional rational subspace of $\R^n$ orthogonal to the $(n-m)$-dimensional rational subspace spanned by the columns of $G$. Therefore, by Lemma 5G and Corollary 5I in \cite{Schmidt_Approximations}, we have $|\det(F)|=|\det(A_{\gamma})|/\gcd(A)$ and, consequently,

%Since $A_\sigma$ is nonsingular,  $L_{\sigma}(A,{\ve b})=\pi_{\sigma}(\Lambda(A,{\ve b}))$ is an $(n-m)$-dimensional affine lattice.
%Furthermore, the lattice $L_{\sigma}(A)= L_{\delta}(A,{\ve 0})$ has determinant
%
\be\label{determinant_via_submatrix}\det(\Lambda(A))=\frac{|\det(A_{\gamma})|}{\gcd(A)}\,.\ee

\begin{thm}\label{nondegenerate}
Suppose that  $A$ satisfies \eqref{matrix_partition}, ${\ve b}\in \Z^m$ and  $\ve x^*$ is given by \eqref{vertex_assumption}. 
For any vertex ${\ve z}^{*}$ of the corner polyhedron $\CP_{\gamma}(A, {\ve b})$
the bound
\be\label{product_det_z}
\prod_{j\in\bar\gamma}(z_j^*+1)\leq \frac{|\det(A_{\gamma})|}{\gcd(A)}
\ee
holds.
\end{thm}

\begin{proof}

Since $\pi_{\gamma} |_{\Sp({A,{\ve b}})}$ is a bijection, the point ${\ve y}^*=\pi_{\gamma}({\ve z}^{*})$ is a vertex of the sail $\sail(\Lambda(A, {\ve b}))$. The result now follows by  Lemma \ref{t:bound_for_irreducible_points} and \eqref{determinant_via_submatrix}.
\end{proof}

\section{Proof of Theorem \ref{t:IPTP_general}}

Theorem \ref{t:IPTP_general} is an immediate consequence of Theorem \ref{nondegenerate} and the following lemma:

\begin{lemma}
\label{proximity_bound_by_product}
Suppose that  $A$ satisfies \eqref{matrix_partition}, and let ${\ve b}\in \Z^m$. Let $\ve x^*$ be a vertex of $\KP(A, {\ve b})$, and let $\gamma$ be any basis of $A$ containing $\supp({\ve x}^{*})$.
Let ${\ve z}^{*}$ be an integral vector satisfying $A{\ve z}^{*}={\ve b}$, with ${\ve z}_{\bar{\gamma}}^{*}\geq \zero$, and let $r=\|{\ve z}^{*}_{\bar\gamma}\|_0$. Then 
\begin{equation}\label{eq:r_equals_0_lemma}
{\ve x}^{*}={\ve z}^{*} \mbox{ if } r= 0\,,
\end{equation}
\begin{equation}\label{eq:r_equals_1_lemma}
\|{\ve x}^{*}-{\ve z}^{*}\|_{\infty} \le \frac{\Sigma(A)}{|\det(A_{\gamma})|}\prod_{j\in\bar\gamma}(z_j^*+1) - 1 \mbox{ if } r= 1 \mbox{ and }
\end{equation}
\begin{equation}\label{eq:r_greater_1_lemma}
\|{\ve x}^{*}-{\ve z}^{*}\|_{\infty} \,\frac{2^{r}}{r} \leq \frac{\Sigma(A)}{|\det(A_{\gamma})|}\prod_{j\in\bar\gamma}(z_j^*+1) \mbox{ if } r\ge 2\,.
\end{equation}
\end{lemma}

\begin{proof}
If $r=\|{\ve z}_{\bar\gamma}^{*}\|_0=0$ the vector ${\ve z}^{*}_{\gamma}$ is the unique solution to the system $A_{\gamma}{\ve x}_{\gamma}= {\ve b}$. Therefore
\eqref{eq:r_equals_0_lemma} holds.

In the rest of the proof we assume that $r\ge 1$.
We will set $\delta= \|{\ve x}^*-{\ve z}^{*}\|_{\infty}$ and consider the following two cases. First suppose that there exists an index $j\in\bar\gamma$ such that
$\delta= |x^*_j-z_j^{*}|=z_j^{*}$.  Observe that $r$ of the numbers $z^{*}_{m+1}, \ldots, z^{*}_{n}$ are nonzero. Hence,
\be\label{distance_nonbasic}
(\delta+1) 2^{r-1}\le \prod_{j\in\bar\gamma} (z^{*}_j+1)
\ee
and so
\be\label{distance_nonbasic_rearranged}
\delta \frac{2^{r}}{r}\le   \frac{2}{r}\prod_{j\in\bar\gamma} (z^{*}_j+1) -\frac{2^{r}}{r}\,.
\ee
Since $\Sigma(A)\geq|\det(A_{\gamma})|$, inequality \eqref{distance_nonbasic_rearranged} justifies both \eqref{eq:r_equals_1_lemma} and \eqref{eq:r_greater_1_lemma}.

Now suppose that $\delta= x^*_j-z^{*}_j$ for $j\in \gamma$. We can write
\bea
A_{\gamma}{\ve z}^{*}_{\gamma}+A_{\bar\gamma}{\ve z}^{*}_{\bar\gamma}={\ve b}\mbox{ and }
A_{\gamma}{\ve x}^*_{\gamma}={\ve b}\,.
\eea
Therefore
\be\label{equation_difference_vector}
A_{\gamma}({\ve x}^*_{\gamma}-{\ve z}^{*}_{\gamma})= A_{\bar\gamma}{\ve z}^{*}_{\bar\gamma}\,.
\ee
Given a vector ${\ve v}\in \R^m$, we will denote by $A_{\gamma}^j({\ve v})$ the matrix obtained from $A_{\gamma}$ by replacing its $j$-{th} column with ${\ve v}$. Let $A_{1}, \ldots, A_{n}$  be the columns of the matrix $A$. Solving \eqref{equation_difference_vector} by Cramer's rule, we have
\be\label{Cramer}
\begin{split}
 \delta=x^*_j-z^{*}_j = \frac{ \det (A_{\gamma}^j(A_{\bar\gamma}{\ve z}^{*}_{\bar\gamma}) )}{\det(A_{\gamma})}\\
 =\frac{1}{\det(A_{\gamma})}(z^{*}_{m+1}\det(A_{\gamma}^j(A_{m+1}) )+ \cdots + z^{*}_{n}\det(A_{\gamma}^j(A_{n})))\,.
\end{split}
\ee
If $r=1$, then for some $i\in \bar\gamma$ we can write
\be\label{Product_r_equals_1}
 \delta=\frac{z^{*}_{i}\det(A_{\gamma}^j(A_{i}) )}{\det(A_{\gamma})}= \left(z^{*}_i + 1\right)\frac{\det(A_{\gamma}^j(A_{i}) )}{\det(A_{\gamma})} - \frac{\det(A_{\gamma}^j(A_{i}) )}{\det(A_{\gamma})}\,.
\ee
Equation \eqref{Product_r_equals_1} plus the integrality of $\delta$ imply \eqref{eq:r_equals_1_lemma}.

To settle the case $r\ge 2$, observe that \eqref{Cramer} implies
\be\label{Delta_upper_bound_r_ge_2}
\delta \leq (z^{*}_{m+1}+\cdots+z^{*}_n)\frac{\Sigma(A)}{|\det(A_{\gamma})|}\,.
\ee

Without loss of generality, assume that $z^{*}_i\neq 0$ for $i\in \{m+1, \ldots, m+r\}$ and $z^{*}_i=0$ for $m+r<i\le n$.
Then, by \eqref{Delta_upper_bound_r_ge_2} and Lemma \ref{t:sum_product}, we have
\begin{equation}
\delta \leq \frac{r(z^{*}_{m+1}+1)\cdots(z^{*}_{m+r}+1)\Sigma(A)}{2^{r}|\det(A_{\gamma})|}\,.
\end{equation}
This establishes inequality \eqref{eq:r_greater_1_lemma}.
\end{proof}
%
%and, noticing that $k\ge r$,
%%
%\be\label{second_intermediate_bound}
%\delta\frac{2^{r}}{r}<
%\frac{k2^{r}}{r2^k}\frac{\Sigma(A)}{\gcd(A)}\le \frac{\Sigma(A)}{\gcd(A)}\,,\ee
%that confirms \eqref{eq:r_greater_1}.

\section{Proof of Theorem \ref{thm:upper_bound_extension}}

As in the proof of Theorem \ref{t:IPTP_general}, we have that Theorem
\ref{thm:upper_bound_extension} is an immediate consequence of Lemma
\ref{proximity_bound_by_product} and the generalization of Theorem
\ref{nondegenerate} given below. Recall that $\tau$ denotes the
support of $\xx^{*}$, and $\CP_{\tau}(A,\bb)$ denotes the polyhedron
\[
\CP_{\tau}(A,\bb)=\conv\left(\left\{ \xx\in\Z^{n}:A\xx=\bb,\;\xx_{\bar{\tau}}\geq\zero\right\} \right).
\]

\begin{thm}
\label{thm:upper_bound_product_coords_extension}Let $\zz^{*}$ be
a vertex of $\CP_{\tau}(A,\bb)$. Then there exists a basis $\gamma$
of $A$ containing $\tau$ such that
\[
\prod_{j\in\ol{\gamma}}(\zz_{j}^{*}+1)\leq r^{d}\,\frac{\left|\det(A_{\gamma})\right|}{\gcd(A)},
\]
where $r=\left\Vert \zz_{\bar{\gamma}}^{*}\right\Vert _{0}$ and $d=m-\left|\tau\right|$.
\end{thm}
Theorem \ref{thm:upper_bound_product_coords_extension} is proved
over the remainder of this section.

\subsection{Convex geometry lemmas}
For an affine subspace $F\subseteq\R^d$, let $\vol_F(\cdot)$ denote the standard Lebesgue measure with respect to $F$. We denote $\vol_{\R^d}(\cdot)$ simply by $\vol_d(\cdot)$.
\begin{lemma}[{Blichfeldt's lemma \cite[Chapter III, Theorem I]{cassels1996introduction}}]
\label{lem:Blichfeldt lemma}Let $K\subseteq\R^{d}$ be bounded,
nonempty, Lebesgue measurable and let $\Lambda$ be a full-dimensional
lattice in $\R^{d}$. Suppose that the difference set $K-K$ contains no nonzero
lattice points from $\Lambda$. Then $\vol_{d}(K)\leq\det(\Lambda)$.\qed
\end{lemma}

\begin{thm}[{Brunn's concavity principle \cite[Theorem 1.2.1]{artstein2015asymptotic}}]
\label{lem:Brunn concavity}Let $K$ be a convex body, and let $F$ be a $k$-dimensional subspace of $\R^d$. Then the function $g:F^{\perp}\rightarrow\R$ defined by
\[
g(x)=\vol_{F+x}(K\cap(F+x))^{1/k}
\]
is concave on its support.
\end{thm}

By a \emph{slab} we mean the nonempty intersection of two halfspaces
with antiparallel normals. Let $\qq\in\R^d$ be nonzero.
The \emph{width} of a set $K\subseteq\R^{d}$ along $\qq$ is defined
to be
\[
w_{\qq}(K):=\left(\sup_{\xx\in K}\qq^{\t}\xx\right)-\left(\inf_{\xx\in K}\qq^{\t}\xx\right).
\]

\begin{prop}
\label{prop:lower_bound_volume_slab}Let $K$ be a centrally symmetric
convex body with centre $\cc$. Let $S$ be a slab centred at $\cc$
with a facet normal $\qq$. If $S$ does not contain $K$, then
\[
\vol_{d}(K\cap S)\geq\frac{w_{\qq}(S)}{w_{\qq}(K)}\cdot\vol_{d}(K).
\]
\end{prop}
\begin{proof}
Without loss of generality, we may assume $\cc$ is the origin. For
$\lambda\in[-1,1]$, define the affine hyperplane
\[
L_{\lambda}:=\left\{ \xx\in\R^{d}:\qq^{\t}\xx=\lambda \cdot w_{\qq}(K)/2\right\} \,.
\]
Let $K_{\lambda}:=K\cap L_{\lambda}$, and define the cross-sectional
volume
\[
f(\lambda):=\vol_{L_{\lambda}}\left(K_{\lambda}\right).
\]
By symmetry, we have $K_{\lambda}=-K_{-\lambda}$. Hence, $f$ is an even function on $[-1,1]$, which means that $g(\lambda):=(f(\lambda))^{1/(d-1)}$ is an even function as well. Since $g$ is concave on $[-1,1]$ by Brunn's concavity principle, we have that $g$, and therefore $f$, is a decreasing function on $[0,1]$. 

Now let $\delta:=w_{\qq}(S)/w_{\qq}(K)$. By Fubini's theorem, symmetry, and monotonicity on $[0,1]$, we
conclude
\[
\vol_{d}(K\cap S)=\int_{-\delta}^{\delta}f(\lambda)d\lambda =2\int_{0}^{\delta}f(\lambda)d\lambda \geq 2\delta\int_{0}^{1}f(\lambda)d\lambda=\frac{w_{\qq}(S)}{w_{\qq}(K)}\cdot\vol_{d}(K).\qedhere
\]
\end{proof}
The notion of irreducibility from Lemma \ref{t:bound_for_irreducible_points}
can be mildly generalized as follows. Let $C$ be a pointed cone.
Let $\Lambda\subset\Z^{d}$ be a $d$-dimensional integer lattice.
The point ${\ve x}\in C\cap\Z^{d}$ is called {\em irreducible (with
respect to $\Lambda$ and $C$)} if
\[
(-\xx+C)\cap(\xx-C)\cap\Lambda=\{\zero\}.
\]

Let $\ve r\in\Z^{d}$ and consider the affine lattice $\Gamma=\ve r+\Lambda$.
We will call the set $\sail(\Gamma,C)=\conv(\Gamma\cap C)$ the {\em
sail} associated with $\Gamma$ and $C$.
\begin{lemma}
\label{t:bound_for_irreducible_points_cone} Every vertex of the sail
$\sail(\Gamma,C)$ is irreducible.
\end{lemma}
\begin{proof}
Let ${\ve x}$ be a vertex of $\sail(\Gamma,C)$. Suppose, to derive
a contradiction, that ${\ve x}$ is reducible. Then there exists nonzero
$\ve{\lambda}\in\Lambda$ and vectors ${\ve y},{\ve y}'\in C$ such
that $\ve{\lambda}=-\xx+\yy=\xx-\yy'$. The fact that $\xx$ is a
vertex of $\sail(\Gamma,C)$ implies $\xx\in\Gamma$, and therefore
both $\yy=\ve{\lambda}+\xx$ and $\yy'=-\ve{\lambda}+\xx$ are contained
in $\Gamma\cap C$, and hence in $\sail(\Gamma,C)$. Since $\ve{\lambda}$
is nonzero, we conclude that $\xx=(\yy+\yy')/2$ is not a vertex of
$\sail(\Gamma,C)$.
\end{proof}

\subsection{Lemmas for Theorem \ref{thm:upper_bound_product_coords_extension}}
In this subsection we further assume the condition that $\bar{\tau}\subseteq\supp(\zz^{*})$. We fix a basis $\gamma$ of $A$ containing $\tau$ such that for all $i\in\gamma\backslash\tau$,
\be
\label{eq:careful_choice_of_gamma}
z_{i}^{*}+1\geq\frac{1}{r}\sum_{j\in\bar{\gamma}}\left|(A_{\gamma}^{-1}A_{\bar{\gamma}})_{i,j}(z_{j}^{*}+1)\right|.
\ee
The existence of such a basis $\gamma$ is justified in Proposition \ref{prop:There_exists_a_basis}. Without loss of generality, we continue with our notational assumption that $\gamma=\left\{ 1,2,\ldots,m\right\} $
and we further assume $\gamma\backslash\tau=\left\{ 1,2,\ldots,d\right\} .$ We denote the rows of
the matrix $-A_{\gamma}^{-1}A_{\bar{\gamma}}$ by $\qq_{1}^{\t},\qq_{2}^{\t},\ldots,\qq_{m}^{\t}$. Note that the equality $A\zz^{*}=\bb$
implies $\qq_{i}^{\t}\zz_{\bar{\gamma}}^{*}=z_{i}^{*}-x_{i}^{*}$ for
all $i\in\gamma$.

\begin{prop}
\label{prop:There_exists_a_basis}Assume $\bar{\tau}\subseteq\supp(\zz^{*})$.
Then there exists a basis $\gamma$ of $A$ containing $\tau$ satisfying inequality \eqref{eq:careful_choice_of_gamma}.
\end{prop}
\begin{proof}
Among all bases of $A$ containing $\tau$, choose a basis $\gamma$
so that the quantity $\left|\det A_{\gamma}\right|\cdot\prod_{i\in\gamma}(z_{i}^{*}+1)$
is as large as possible. If $i\in\gamma$ and $j\in\bar{\gamma}$, then by Cramer's rule we
have
\[
q_{i,j}=-\frac{\det(A_{\gamma}^{i}(A_{j}))}{\det(A_{\gamma})},
\]
where $A_{\gamma}^{i}(A_{j})$ denotes the matrix obtained by replacing
column $i$ of $A_{\gamma}$ with column $j$ of $A$. The choice
of $\gamma$ implies that if $i\in\gamma\backslash\tau$ and $j\in\bar{\gamma},$
then
\[
z_{i}^{*}+1\geq\left|\frac{\det(A_{\gamma}^{i}(A_{j}))}{\det(A_{\gamma})}(z_{j}^{*}+1)\right|=\left|q_{i,j}(z_{j}^{*}+1)\right|.
\]
The condition $\bar{\tau}\subseteq\supp(\zz^{*})$ implies $r=\left|\bar{\gamma}\right|$.
Hence, for all $i\in\gamma\backslash\tau$, we have
\[
z_{i}^{*}+1\geq\frac{1}{r}\sum_{j\in\bar{\gamma}}\left|q_{i,j}(z_{j}^{*}+1)\right|.\qedhere
\]
\end{proof}

Let $\oneone_{n-m}\in\R^{n-m}$ be the vector of all ones, and define, for
each $i\in\gamma\backslash\tau$,
\begin{align*}
S_{i} & :=\left\{ \xx\in\R^{n-m}:-\tfrac{1}{2}<\qq_{i}^{\t}\xx<\qq_{i}^{\t}\zz_{\bar{\gamma}}^{*}+\tfrac{1}{2}\right\} .
\end{align*}
Also define
\begin{align*}
B & :=\left\{ \xx\in\R^{n-m}:-\tfrac{1}{2}\oneone_{n-m}<\xx<\zz_{\bar{\gamma}}^{*}+\tfrac{1}{2}\oneone_{n-m}\right\},
\end{align*}
and for each $i\in\gamma\backslash\tau$, let $P_{i}:=P_{i-1}\cap S_{i}$ with $P_{0}=B$. Let $P:=P_d$.
\begin{lemma}
\label{lem:lower_bound_volume}Assume $\bar{\tau}\subseteq\supp(\zz^{*})$.
Then
\[
\vol_{n-m}(P)\geq\frac{1}{r^{d}}\prod_{j\in\bar{\gamma}}(z_{j}^{*}+1).
\]
\end{lemma}
\begin{proof}
Suppose $i\in\gamma \backslash \tau$. If $S_{i}$ contains $P_{i-1}$ then $P_{i-1}=P_{i}$,
and hence
\[
\vol_{n-m}(P_{i})=\vol_{n-m}(P_{i-1}).
\]
Otherwise, define
\[
\lambda_{i}:=\frac{w_{\qq_{i}}(S_{i})}{w_{\qq_{i}}(P_{i-1})}.
\]
The fact that $\bar{\tau}\subseteq\supp(\zz^{*})$ implies $z_{i}^{*}\geq1$
for all $i\in\gamma\backslash\tau$. Applying Proposition \ref{prop:There_exists_a_basis},
we get
\[
\lambda_{i}\geq\frac{w_{\qq_{i}}(S_{i})}{w_{\qq_{i}}(B)}=\frac{\qq_{i}^{\t}\zz_{\bar{\gamma}}^{*}+1}{\sum_{j\in\bar{\gamma}}\left|q_{i,j}(z_{j}^{*}+1)\right|}=\frac{z_{i}^{*}+1}{\sum_{j\in\bar{\gamma}}\left|q_{i,j}(z_{j}^{*}+1)\right|}\geq\frac{z_{i}^{*}+1}{r(z_{i}^{*}+1)}=\frac{1}{r}.
\]
Since $S_{i}$ does not contain $P_{i-1}$, Proposition \ref{prop:lower_bound_volume_slab}
applies, and so
\[
\vol_{n-m}(P_{i})\geq\frac{w_{\qq_{i}}(S_{i})}{w_{\qq_{i}}(P_{i-1})}\vol_{n-m}(P_{i-1})=\lambda_{i}\vol_{n-m}(P_{i-1})\geq\frac{1}{r}\vol_{n-m}(P_{i-1}).
\]
Applying induction to the sequence of polytopes $P=P_d,\ldots,P_1,P_0=B$, we get
\[
\vol_{n-m}(P)\geq\frac{1}{r^{d}}\vol_{n-m}(B)=\frac{1}{r^{d}}\prod_{j\in\bar{\gamma}}(z_{j}^{*}+1).\qedhere
\]
\end{proof}
\begin{lemma}
\label{lem:upper_bound_polyhedron_volume}Assume $\bar{\tau}\subseteq\supp(\zz^{*})$.
Then
\[
\vol_{n-m}(P)\leq\frac{\left|\det(A_{\gamma})\right|}{\gcd(A)}.
\]
\end{lemma}
\begin{proof}
Recall we defined the lattice $\Lambda(A)=\pi_{\gamma}(\ker(A)\cap\Z^{n})$,
whose determinant is given by $\left|\det(A_{\gamma})\right|/\gcd(A)$
by \eqref{determinant_via_submatrix}. We show that $(P-P)\cap\Lambda(A)=\{\zero\}$.
The conclusion then follows from Lemma \ref{lem:Blichfeldt lemma}.

Suppose that $\uu,\vv\in P$ and $\uu-\vv\in\Lambda(A)$. Since $P$
is symmetric, $P-P$ is the origin-symmetric translate of $2P$, and
therefore
\begin{equation}
\begin{gathered}-\zz_{\bar{\gamma}}^{*}-\one_{n-m}<\uu-\vv<\zz_{\bar{\gamma}}^{*}+\one_{n-m}\\
-\qq_{i}^{\t}\zz_{\bar{\gamma}}^{*}-1<\qq_{i}^{\t}(\uu-\vv)<\qq_{i}^{\t}\zz_{\bar{\gamma}}^{*}+1\text{ for all \ensuremath{i\in\gamma\backslash\tau}}.
\end{gathered}
\label{eq:x_minus_y_open_bounds}
\end{equation}
The lattice $\Lambda(A)$ can be characterized as the set of points
$\xx\in\Z^{n-m}$ such that $\qq_{i}^{\t}\xx\in\Z$ for
each $i\in\gamma$. Hence, the inequalities from (\ref{eq:x_minus_y_open_bounds})
imply
\[
\begin{gathered}-\zz_{\bar{\gamma}}^{*}\leq\uu-\vv\leq\zz_{\bar{\gamma}}^{*}\\
-\qq_{i}^{\t}\zz_{\bar{\gamma}}^{*}\leq\qq_{i}^{\t}(\uu-\vv)\leq\qq_{i}^{\t}\zz_{\bar{\gamma}}^{*}\text{ for all \ensuremath{i\in\gamma\backslash\tau}}.
\end{gathered}
\]
In particular, $\uu-\vv$ lies in the polyhedron $(-\zz_{\bar{\gamma}}^{*}+C)\cap(\zz_{\bar{\gamma}}^{*}-C)$,
where
\[
C:=\left\{ \xx\in\R^{n-m}:\xx\geq0,\;\qq_{i}^{\t}\xx\geq0\text{ for all \ensuremath{i\in\gamma\backslash\tau}}\right\} .
\]
By assumption, $\zz_{\bar{\gamma}}^{*}$ is a vertex of the sail $\sail(\Lambda(A,\bb),C)$.
Hence, $\zz_{\bar{\gamma}}^{*}$ is irreducible by Lemma \ref{t:bound_for_irreducible_points_cone},
and therefore $\uu=\vv$.
\end{proof}

\subsection{Proof of Theorem \ref{thm:upper_bound_product_coords_extension}}

Let $\mu=\tau\cup\supp(\zz^{*})$, which we may assume without loss of generality is given by $\mu=\{1,2,\ldots,\left|\mu\right|\}$, and let $A'_{\mu}$ be any full
row rank integer matrix with the same rowspace as $A_{\mu}$. We have
that $\xx_{\mu}^{*}$ is a basic feasible solution of the system 
\be
A_{\mu}'\xx_{\mu}=\bb',\;\xx_{\mu}\geq\zero,
\label{eq:reduced_system}
\ee
where $\bb':=A_{\mu}'\xx_{\mu}^{*}$. Moreover, letting
\[
\CP_{\tau}(A_{\mu}',\bb') := \conv(\{\xx_{\mu}\in\Z^{\left|\mu\right|}:A_{\mu}'\xx_{\mu}=\bb',\;\xx_{\mu\backslash\tau}\geq\zero\}),
\]
we have that $\CP_{\tau}(A_{\mu}',\bb')\times\{\zero\}^{\bar{\mu}}$
is the face of $\CP_{\tau}(A,\bb)$ for which the constraints $\xx_{\bar{\mu}}\geq\zero$ are tight, and this face contains $\zz^{*}=(\zz_{\mu}^{*},\zero)$. Hence, $\zz_{\mu}^{*}$ is a vertex of $\CP_{\tau}(A_{\mu}',\bb')$.
We may therefore apply the above results to $\xx_{\mu}^{*},\zz_{\mu}^{*},$
and the system \eqref{eq:reduced_system}. Let $\sigma\subseteq\mu$
be a basis of $A_{\mu}'$ containing $\tau$ satisfying \eqref{eq:careful_choice_of_gamma}, whose existence is guaranteed by \eqref{eq:careful_choice_of_gamma}. Then Lemmas \ref{lem:lower_bound_volume}
and \ref{lem:upper_bound_polyhedron_volume} imply
\[
\prod_{j\in\bar{\sigma}}(\zz_{j}^{*}+1)\leq r^{\left|\sigma\right|-\left|\tau\right|}\,\frac{\left|\det(A_{\sigma}')\right|}{\gcd(A_{\mu}')}.
\]
Now let $\gamma$ be a basis of $A$ containing $\sigma$. Then $\mu$
and $\gamma\backslash\sigma$ partition $\mu\cup\gamma$. Up to invertible
row operations, we can write
\[
A_{\mu\cup\gamma}=\begin{pmatrix}A_{\mu}' & A_{\gamma\backslash\sigma}'\\
\zero & A_{\gamma\backslash\sigma}''
\end{pmatrix}=\begin{pmatrix}A_{\mu\backslash\sigma}' & A_{\sigma}' & A_{\gamma\backslash\sigma}'\\
\zero & \zero & A_{\gamma\backslash\sigma}''
\end{pmatrix},
\]
where both $A_{\sigma}'$ and $A_{\gamma\backslash\sigma}''$ are
both invertible. Now, every nonzero maximal subdeterminant of $A_{\mu\cup\gamma}$
is the product of $\det(A_{\gamma\backslash\sigma}'')$ with a maximal
subdeterminant of $A_{\mu}'$. It follows that
\[
\gcd(A_{\mu\cup\gamma})=\bigl|\det(A_{\gamma\backslash\sigma}'')\bigr|\cdot\gcd(A_{\mu}'),
\]
and hence
\[
\frac{\bigl|\det(A_{\sigma}')\bigr|}{\gcd(A_{\mu}')}=\frac{\bigl|\det(A_{\sigma}')\bigr|\cdot\bigl|\det(A_{\gamma\backslash\sigma}'')\bigr|}{\gcd(A_{\mu\cup\gamma})}=\frac{\bigl|\det(A_{\gamma})\bigr|}{\gcd(A_{\mu\cup\gamma})}.
\]
We conclude
\[
\prod_{j\in\bar{\gamma}}(\zz_{j}^{*}+1)\leq\prod_{j\in\bar{\sigma}}(\zz_{j}^{*}+1)\leq r^{\left|\sigma\right|-\left|\tau\right|}\,\frac{\bigl|\det(A_{\gamma})\bigr|}{\gcd(A_{\mu\cup\gamma})}\leq r^{d}\,\frac{\bigl|\det(A_{\gamma})\bigr|}{\gcd(A)}.\qedhere
\]

\section{Proof of Theorem \ref{t:IPTP_positive_knapsack}}

First we will show that the knapsack polytope $\KP({\ve a}, b)$ contains a vertex of the corner polyhedron
$\CP_{\gamma}({\ve a}, {b}) $.  Let ${\ve z}^*$ be a vertex of $\CP_{\gamma}({\ve a}, {b}) $ that gives an optimal solution to the linear program 
\bea \max\{x_1: {\ve x}=(x_1, \ldots, x_n)^{\t} \in \CP_{\gamma}({\ve a}, {b}) \}\,.\eea
By definition of $\CP_{\gamma}({\ve a}, {b})$ the vertex ${\ve z}^*$ is in $\KP({\ve a}, b)$ if and only if $z_1^*\ge 0$.
Since $\KP({\ve a}, b)\subset \CP_{\gamma}({\ve a}, {b}) $, it is now sufficient to choose any integer point  ${\ve z}=(z_1, \ldots, z_n)^{\t}\in \KP({\ve a}, b)$ and observe that 
$z_1^*\ge z_1\ge 0$.

Applying Theorem  \ref{t:IPTP_general} with the vertex ${\ve z}^*\in \KP({\ve a}, b)$ we immediately  obtain \eqref{eq:r_equals_0_knapsack}
and \eqref{eq:r_equals_1_knapsack}. Further,  the bound \eqref{eq:r_greater_1} implies for $r\ge 2$ the non-strict inequality 
\begin{equation}\label{eq:r_greater_1_knapsack_nonstrict}
\|{\ve x}^{*}-{\ve z}^{*}\|_{\infty} \,\frac{2^{r}}{r} \le  \|{\ve a}\|_{\infty}\,.
\end{equation}
To show that \eqref{eq:r_greater_1_knapsack_nonstrict} is strict (and hence that \eqref {eq:r_greater_1_knapsack} holds), it is sufficient to prove that the bound 
\eqref{Delta_upper_bound_r_ge_2} in the proof of Theorem \ref{t:IPTP_general}  is strict in the knapsack scenario.
Specifically, we need to prove that for the vertex ${\ve z}^*$ 
\be\label{lincomb_via_sum_knapsack}
\delta= |x^*_1-z^*_1|< \frac{(z_2^*+\cdots+z^*_{n})\|{\ve a}\|_{\infty}}{a_1}\,.
\ee

Set $A=(a_1, \ldots, a_n)\in\Z^{1\times n}$  and consider the affine lattice $\Lambda({\ve a},{b}):=\Lambda(A,{b})$.
We can write
\begin{equation}\label{affine_congruence}
\Lambda({\ve a}, b)= \{(\lambda_2, \ldots, \lambda_n)^{\t}\in \Z^{n-1}: \lambda_2 a_2+\cdots+ \lambda_{n-1}a_{n-1}\equiv b\, (\modulo a_1)\}\,.
\end{equation}
Following \eqref{structure_of_projection},  the map $\pi_{\gamma} |_{\Sp({A,{b}})}$, with $\gamma=\{1\}$ in the knapsack scenario, is a bijection. It follows that the point ${\ve y}^*=\pi_{\gamma}({\ve z}^{*})$ is a vertex of the sail $\sail(\Lambda({\ve a}, {b}))$.

Suppose, to derive a contradiction, that the equality 
\be\label{contr_equal} \delta=\frac{(z_2^*+\cdots+z_{n}^*)\|{\ve a}\|_{\infty}}{a_1}\ee
holds. By \eqref{Cramer} we have
\bea
 \delta=\frac{z^{*}_{2}a_2+ \cdots + z^{*}_{n}a_n}{a_1}
\eea
 and, consequently, \eqref{contr_equal}  implies $a_2=\cdots=a_n=\|{\ve a}\|_{\infty}$. Therefore, using \eqref{affine_congruence}, the affine lattice $\Lambda({\ve a}, b)$ contains the points 
\be \label{simplex}
(z_2^*+\cdots+z_{n}^*){\ve e}_j\,,\; j\in \{1,\ldots, n-1\}\,.
\ee 
The point ${\ve y}=(z_2^*, \ldots, z^*_{n})^{\t}$, in its turn, belongs to the simplex with vertices \eqref{simplex} and has $\|{\ve y}\|_0=r\ge 2$.
Therefore ${\ve y}$ cannot be a vertex of the sail $\sail(\Lambda({\ve a},{b}))$. The derived contradiction completes the proof of Theorem \ref{t:IPTP_positive_knapsack}.

\section{Proof of Corollary \ref{coro_upper_bound}}

By Theorem \ref{t:IPTP_positive_knapsack} the knapsack polytope $\KP({\ve a}, b)$ contains a vertex ${\ve z}^{*}$ of $\CP_{\gamma}({\ve a}, {b}) $.
Therefore
\be\label{gap-distance}
IG({\ve c}, {\ve a}, {b})\le \|{\ve x}^{*}-{\ve z}^{*}\|_{\infty}\sum_{i\in\supp({\ve x}^{*}-{\ve z}^{*})}|c_i|\,.
\ee
If $r=0$ we have ${\ve x}^*={\ve z}^{*}$ that justifies \eqref{eq:r_equals_0_gap}.
Further,  \eqref{gap-distance} implies  the bound
\bea
IG({\ve c}, {\ve a}, {b})\le (r+1)\|{\ve x}^{*}-{\ve z}^{*}\|_{\infty}\|{\ve c}\|_{\infty}
\eea
that immediately gives \eqref{eq:r_equals_1_gap} and \eqref{eq:r_greater_1_gap}.

\section{Proof of Theorem \ref{t:IPTP_positive_knapsack_sharpness}}

For  $s\ge 2$ set
$
{\ve a}^{(s)}=(2^{s-1}, 2^{s-2}, \ldots,  1)^{\t}
$
and
$b^{(s)}= {\ve 1}^{\t}_s{\ve a}^{(s)}=2^{s}-1$. Let $\KP_I({\ve a}^{(s)}, {b}^{(s)})= \conv (\KP({\ve a}^{(s)}, {b}^{(s)})\cap \Z^s)$ be the {integer hull} of the knapsack polytope $\KP({\ve a}^{(s)}, b^{(s)})$.

We will need the following observations. 
\begin{lemma}\label{vertex_integer_hull}
The point ${\ve 1}_{s}$ is a vertex of the polytope $\KP_I({\ve a}^{(s)}, {b}^{(s)})$.
\end{lemma}
\begin{proof}
We will use induction on $s$. The basis step $s=2$ holds as there are only two integer points ${\ve 1}_2$ and $(0,3)^{\t}$
%six integer points ${\ve 1}_3$, $(3, 0, 1)^{\t}$, $(1,3, 0)^{\t}$, $(3,2,0)^{\t}$, $(5,1,0)^{\t}$ and $(7,0,0)^{\t}$
%in the polytope $\KP({\ve a}^{(3)}, {b}^{(3)})$.
in the polytope $\KP({\ve a}^{(2)}, {b}^{(2)})$.
To verify the inductive step, suppose that the result does not hold for some $s\ge 3$. Observe that any integer point ${\ve z}=(z_1, \ldots, z_s)^{\t}\in \KP({\ve a}^{(s)}, {b}^{(s)})$ has $z_1\le 1$. Consequently, ${\ve 1}_s$ belongs to the face $\KP_I({\ve a}^{(s)}, {b}^{(s)})\cap\{{\ve x}\in \R^s: x_1=1\}$ of the polyhedron $\KP_I({\ve a}^{(s)}, {b}^{(s)})$. Hence ${\ve 1}_s$ is a convex combination of some integer points in $\KP({\ve a}^{(s)}, { b}^{(s)})$ that have the first entry $1$.  Therefore, removing the first entry we obtain a convex combination of integer points from $\KP({\ve a}^{(s-1)}, {\ve b}^{(s-1)})$ equal to  ${\ve 1}_{s-1}$. The obtained contradiction completes the proof. 
\end{proof}

For the rest of the proof we assume $s\ge 3$.
\begin{lemma}\label{vertex_sail}
The point ${\ve 1}_{s-1}$ is a vertex of the sail $\sail(\Lambda({\ve a}^{(s)}, {\ve b}^{(s)}))$.
\end{lemma}
\begin{proof}

Using  \eqref{affine_congruence}, the affine lattice $\Lambda({\ve a}^{(s)}, b^{(s)})$ can be written as
\begin{equation*}\label{affine_congruence_sharpness}
\Lambda({\ve a}^{(s)}, b^{(s)})= \{{\ve x}\in \Z^{s-1}:  2^{s-2}x_2+ \cdots+ x_{s}\equiv -1\, (\modulo 2^{s-1})\}\,.
\end{equation*}
Therefore
\bea \label{simplex_boundary}
{\mathcal H}= \{{\ve x}\in \R^{s-1}: 2^{s-2}x_2+ \cdots+ x_{s}= 2^{s-1}-1\}\,
\eea
is a supporting hyperplane of $\sail(\Lambda({\ve a}^{(s)}, { b}^{(s)}))$. Consequently,
\bea\KP_I({\ve a}^{(s-1)}, {b}^{(s-1)})={\mathcal H}\cap \sail(\Lambda({\ve a}^{(s)}, { b}^{(s)}))
\eea
is a face of $\sail(\Lambda({\ve a}^{(s)}, { b}^{(s)}))$.
The result now follows by Lemma \ref{vertex_integer_hull}.

\end{proof}

For a positive integer $t$ set
\bea
{\ve a}^{(s)}(t)= (a^{(s)}_1(t), \ldots, a^{(s)}_s(t))^{\t}=(2^{s-1}, 2^{s-2}+t2^{s-1},\ldots, 1+t 2^{s-1})^{\t}
\eea
and
$b^{(s)}(t)= {\ve 1}^{\t}_s{\ve a}^{(s)}(t)=2^{s} + (s-1)t2^{s-1}-1$. Consider the vertex ${\ve v}^{(s)}(t)=(b^{(s)}(t)/a_1^{(s)}(t)){\ve e}_1$ of the knapsack polytope $\KP({\ve a}^{(s)}(t), b^{(s)}(t))$.

In view of  \eqref{affine_congruence}, we have $\Lambda({\ve a}^{(s)}, b^{(s)})= \Lambda({\ve a}^{(s)}(t), b^{(s)}(t))$. Therefore, by Lemma \ref{vertex_sail}
the point ${\ve 1}_{s-1}$ is a vertex of the sail $\sail(\Lambda({\ve a}^{(s)}(t), { b}^{(s)}(t)))$. Observe that the sail $\sail(\Lambda({\ve a}^{(s)}(t), { b}^{(s)}(t)))$ is the image of the corner polyhedron $\CP_{{\ve v}^{(s)}(t)}({\ve a}^{(s)}, b^{(s)})$ under the bijective linear map $\pi_{\gamma}|_{\Sp({\ve a}^{(s)}(t), b^{(s)}(t))}(\cdot)$. Using \eqref{structure_of_projection}, the point
\bea
{\ve 1}_s=\pi^{-1}_{\gamma}|_{\Sp({\ve a}^{(s)}(t), b^{(s)}(t))}({\ve 1}_{s-1})
\eea
is a feasible vertex of $\CP_{{\ve v}^{(s)}(t)}({\ve a}^{(s)}, b^{(s)})$.
Note also that  ${\ve 1}_s\in \KP({\ve a}^{(s)}(t), {b}^{(s)}(t))$.

It is now sufficient to show that  for any $\epsilon>0$ 
\be\label{sharp_a_}
\|{\ve v}^{(s)}(t)-{\ve 1}_s\|_{\infty}\frac{2^{s-1}}{s-1}>(1-\epsilon) \|{\ve a}^{(s)}(t)\|_{\infty}
\ee
for sufficiently large $t$. We have
\bea\label{sharp_a_1}
\|{\ve v}^{(s)}(t)-{\ve 1}_s\|_{\infty}=\left|\frac{b^{(s)}(t)}{a_1^{(s)}(t)}-1\right | = (s-1)t+1-\frac{1}{2^{s-1}}\,.
\eea
Finally,
\bea
\frac{\|{\ve v}^{(s)}(t)-{\ve 1}_s\|_{\infty}}{\|{\ve a}^{(s)}(t)\|_{\infty}}= \frac{(s-1)t+1-2^{-(s-1)}}{2^{s-2}+t2^{s-1}}\longrightarrow\frac{s-1}{2^{s-1}}
\eea
as $t\rightarrow\infty$, that implies (\ref{sharp_a_}).

\section{Proof of Theorem \ref{thm:UpperBound}}

We will apply the following result by Bombieri and Vaaler \cite{BombVaal}.

\begin{thm}[{\cite[Theorem~2]{BombVaal}}]\label{eq:sl_sl_f}
Let $A\in \Z^{m\times n}$, $m<n$, be a matrix of rank $m$. There exist $n - m$ linearly independent integral vectors ${\ve y}_1, \ldots,{\ve y}_{n-m} \in \Gamma(A)$ satisfying
\begin{equation*}
 \prod_{i=1}^{n-m} \| {\ve y}_i\|_{\infty} \le \frac{\sqrt{{\rm det}(AA^{T})}}{\gcd(A)}\,.
 \end{equation*}
\end{thm}

Let ${\ve z}^*$ be a vertex of the integer hull $\KP_I(A, {\ve b})$ that gives an optimal solution to \eqref{eq:mainProblem}.
We will show that ${\ve z}^*$  satisfies \eqref{t:Bound_via_transference}.
First, we argue that it suffices to consider the case $\|\ve z^*\|_0=n$.
Suppose that $\|\ve z^*\|_0<n$. For $\tau=\supp(\ve z^*)$ set 
$\bar A=A_\tau$, $\bar{\ve b}=\ve b$, $\bar{\ve c}=\ve c_\tau$, and $\bar{\ve z}^*=\ve z^*_\tau$.
By removing linearly dependent rows, we may assume that $\bar A$ has full row rank. Let $\bar m = \rank(\bar A) \le m$.
%We will assume that $\rank(\bar A)=m$, otherwise we remove redundant constraints from $\bar A \bar{\ve x}=\bar{\ve b}$.
%Thus $\bar A$ has a linearly independent rows.
Observe that  $\bar{\ve z}^*$ is an optimal solution for the corresponding problem \eqref{eq:mainProblem} with minimal support.
Furthermore, note that $\bar{\ve z}^*$ has full support. 
Now, if \eqref{t:Bound_via_transference} holds true for $\bar{\ve z}^*$, then 
\be\label{split_rho_1}
(\rho({\ve z}^*)+1)^{s-m}\le (\rho(\bar{\ve z}^*)+1)^{s-\bar m}\le \frac{\sqrt{\det(\bar A\bar A^{\t})}}{\gcd(\bar A)}\,.
\ee
Further, using \cite[Lemma~2.3]{ADON2017} we have
\be\label{split_rho_2}
\frac{\sqrt{\det(\bar A\bar A^{\t})}}{\gcd(\bar A)}\le \frac{\sqrt{\det(AA^{\t})}}{\gcd(A)}\,.
\ee
Combining \eqref{split_rho_1} and \eqref{split_rho_2}, we obtain \eqref{t:Bound_via_transference}.

From now on, we may assume that $\|\ve z^*\|_0=n$. Suppose, to derive a contradiction, that \eqref{t:Bound_via_transference} does not holds, that is $(\rho({\ve z}^*)+1)^{n-m}> (\gcd(A))^{-1}\sqrt{\det(AA^{\t})}$. By Theorem \ref{eq:sl_sl_f}, there exists a vector $\ve y\in\Z^n\setminus\{\ve 0\}$ such that
$$A\ve y=\ve 0 \;\; \text{ and } \;\; \|\ve y\|_\infty\le\left(\frac{\sqrt{\det(AA^{\t})}}{\gcd(A)}\right)^{\frac{1}{n-m}}<\rho({\ve z}^*)+1.$$
It follows that both $\ve z^* + \ve y$ and $\ve z^* - \ve y$ are in the knapsack polytope $\KP(A, {\ve b})$. Therefore $\ve z^*$ is not a vertex of $\KP_I(A, {\ve b})$.
The obtained contradiction completes the proof.

%\small
\bibliographystyle{plain}
\bibliography{bib}

\begin{thebibliography}{10}

\bibitem{AADO}
I.~Aliev, G.~Averkov, J.~A. De~Loera, and T.~Oertel.
\newblock Optimizing sparsity over lattices and semigroups.
\newblock {\em IPCO, Lecture Notes in Comput. Sci.}, 12125:40--51, 2020.

\bibitem{Support}
I.~Aliev, J.~A. De~Loera, F.~Eisenbrand, T.~Oertel, and R.~Weismantel.
\newblock The support of integer optimal solutions.
\newblock {\em SIAM J. Optim.}, 28(3):2152--2157, 2018.

\bibitem{ADON2017}
I.~Aliev, J.~A. De~Loera, T.~Oertel, and C.~O'Neill.
\newblock Sparse solutions of linear {D}iophantine equations.
\newblock {\em SIAM J. Appl. Algebra Geom.}, 1(1):239--253, 2017.

\bibitem{AHO}
I.~Aliev, M.~Henk, and T.~Oertel.
\newblock Distances to lattice points in knapsack polyhedra.
\newblock {\em Math. Program.}, 182(1-2, Ser. A):175--198, 2020.

\bibitem{artstein2015asymptotic}
S.~Artstein-Avidan, A.~Giannopoulos, and V.D. Milman.
\newblock {\em Asymptotic Geometric Analysis, Part I}.
\newblock Mathematical Surveys and Monographs. American Mathematical Society,
  2015.

\bibitem{CS_survey}
H.~Boche, R.~Calderbank, G.~Kutyniok, and J.~Vyb\'{\i}ral.
\newblock A survey of compressed sensing.
\newblock In {\em Compressed sensing and its applications}, Appl. Numer.
  Harmon. Anal., pages 1--39. Birkh\"{a}user/Springer, Cham, 2015.

\bibitem{BombVaal}
E.~Bombieri and J.~Vaaler.
\newblock On {S}iegel's lemma.
\newblock {\em Invent. Math.}, 73(1):11--32, 1983.

\bibitem{candestao}
E.~J. Candes and T.~Tao.
\newblock Decoding by linear programming.
\newblock {\em IEEE Trans. Inform. Theory}, 51(12):4203--4215, 2005.

\bibitem{cassels1996introduction}
J.W.S. Cassels.
\newblock {\em An Introduction to the Geometry of Numbers}.
\newblock Classics in Mathematics. Springer Berlin Heidelberg, 1996.

\bibitem{MR830593}
W.~Cook, J.~Fonlupt, and A.~Schrijver.
\newblock An integer analogue of {C}arath\'eodory's theorem.
\newblock {\em J. Combin. Theory Ser. B}, 40(1):63--70, 1986.

\bibitem{MR839604}
W.~Cook, A.~M.~H. Gerards, A.~Schrijver, and \'{E}. Tardos.
\newblock Sensitivity theorems in integer linear programming.
\newblock {\em Math. Programming}, 34(3):251--264, 1986.

\bibitem{EisenbrandShmonin2006}
F.~Eisenbrand and G.~Shmonin.
\newblock Carath\'eodory bounds for integer cones.
\newblock {\em Oper. Res. Lett.}, 34(5):564--568, 2006.

\bibitem{MR3775840}
F.~Eisenbrand and R.~Weismantel.
\newblock Proximity results and faster algorithms for integer programming using
  the {S}teinitz lemma.
\newblock In {\em Proceedings of the {T}wenty-{N}inth {A}nnual {ACM}-{SIAM}
  {S}ymposium on {D}iscrete {A}lgorithms}, pages 808--816. SIAM, Philadelphia,
  PA, 2018.

\bibitem{Gomory_polyhedra}
R.~E. Gomory.
\newblock Some polyhedra related to combinatorial problems.
\newblock {\em Linear Algebra Appl.}, 2:451--558, 1969.

\bibitem{LPSX}
I.~Lee, J.~Paat, Stallknecht I., and L.~Xu.
\newblock Improving proximity bounds using sparsity.
\newblock {\em arXiv:2001.04659}.

\bibitem{Schmidt_Approximations}
W.~M. Schmidt.
\newblock {\em Diophantine approximations and Diophantine equations}.
\newblock Lecture Notes in Mathematics. Springer-Verlag, 1991.

\bibitem{Sebo}
A.~Seb\H{o}.
\newblock {H}ilbert bases, {C}arath\'eodory's theorem and combinatorial
  optimization.
\newblock In {\em Proceedings of the 1st Integer Programming and Combinatorial
  Optimization Conference}, pages 431--455. University of Waterloo Press, 1990.

\bibitem{thomas2005structure}
Rekha~R Thomas.
\newblock The structure of group relaxations.
\newblock {\em Handbooks in Operations Research and Management Science},
  12:123--170, 2005.

\end{thebibliography}

\end{document}